\documentclass{amsart}

\usepackage{amsfonts}
\usepackage{amssymb}
\usepackage{amsthm}
\usepackage{eucal}
\usepackage{graphicx}
\usepackage{tikz}
\usetikzlibrary{decorations.pathreplacing}

\theoremstyle{plain}
\newtheorem{theorem}{Theorem}[section]
\newtheorem{lemma}[theorem]{Lemma}
\newtheorem{proposition}[theorem]{Proposition}
\newtheorem{corollary}[theorem]{Corollary}
\newtheorem{conjecture}{Conjecture}

\theoremstyle{definition}

\theoremstyle{remark}

\newcommand{\R}{\mathbb{R}}
\newcommand{\Q}{\mathbb{Q}}
\newcommand{\Z}{\mathbb{Z}}
\newcommand{\G}{\Gamma}
\newcommand{\gr}{\mathrm{gr}}
\newcommand{\diam}{\mathrm{diam}}

\newcommand{\D}{\mathcal{D}}
\newcommand{\Cay}{\mathrm{Cay}}
\newcommand{\Aut}{\mathrm{Aut}}

\begin{document}
\title{Self $2$-distance graphs}
\author{A. Azimi}
\author{M. Farrokhi D. G.}

\address{Department of Pure Mathematics, Ferdowsi University of Mashhad, Mashhad, Iran.}
\email{ali.azimi61@gmail.com}
\address{Mathematical Science Research Unit, College of Liberal Arts, Muroran Institute of Technology, 27-1, Mizumoto, Muroran 050-8585, Hokkaido, Japan.}
\email{m.farrokhi.d.g@gmail.com}


\begin{abstract}
All finite simple self $2$-distance graphs with no $4$-cycle, diamond, or triangles with a common vertex are determined. Utilizing these results, it is shown that there is no cubic self $2$-distance graphs.
\end{abstract}
\maketitle
\section{Introduction}
Let $(X,\rho)$ be a metric space and $D$ be a set of positive real numbers. The \textit{distance graph} $G(X,D)$ of $X$ with respect to a distance set $D$ is the graph whose vertex set is $X$ and two distinct vertices $x$ and $y$ are adjacent if $\rho(x,y)\in D$. 

The well-known unit distance graph $G(\R^2,\{1\})$ is the first instance of a distance graph arising from a question of Edward Nelson about its chromatic number in 1950 (see \cite[Chapter 3]{as}). It is shown by Nelson and Isbell \cite{ji-en}, Moser and Moser \cite{lm-wm} and Hadwiger, Debrunner and Klee \cite{hh-hd-vk} that the chromatic number of this graph is between $4$ and $7$. Unit distance graphs are also investigated on any of the sets $\R^n$, $\Q^n$ and $\Z^n$ as well (see \cite{as} for a detailed history). The other well-studied sort of distance graphs are the distance graphs $G(\Z,D)$ introduced by Eggleton, Erd\"{o}s and Skiltons in \cite{rbe-pe-dks}, where $D$ is a set of positive integers. Clearly, every graph $\G$ with associated distance function $d$ defines a metric space $(\G,d)$. Hence, we may define the distance graphs of the graph $\G$ with respect to a set of positive integer distances. For example, the $n$th power of a graph $\G$ is defined simply as the distance graph $G(\G,\{1,\ldots,n\})$. We refer the interested reader to the survey articles \cite{kbc,ddl,fk-hk} for further details concerning the mentioned three kinds of distance graphs, respectively. 

The \textit{$n$th distance graph} (or \textit{$n$-distance graph}) of a graph $\G$ is defined simply as $\G_n:=G(V(\G),\{n\})$. The study of $n$th distance graph initiated by  Simi\'{c} \cite{sks} while solving the graph equation $\G_n\cong L(\G)$, where $L(\G)$ is line graph of $\G$. Regarding the same problem, we have classified of all graphs whose $2$-distance graphs are path or cycle in \cite{aa-mfdg}.

A graph is said to be \textit{self $n$-distance graph} if it is isomorphic to its $n$-distance graph. The aim of this paper is to investigate self $2$-distance graphs under some conditions. More precisely, we will show that self $2$-distance graphs with no squares or disjoin triangles are either odd cycles of order $\geq5$ or the edge product $C_5|C_3$. Also, we show that a self $2$-distance graph with no diamond is either an odd cycle of order $\geq5$, the edged product $C_5|C_3$, or it is isomorphic to one of graphs in Figures 5.1.1 or 5.1.2. One note that our knowledge about $n$-distance graphs can be used to answer/pose some problems in groups through their Cayley graphs. Indeed, we may observe that the $n$th distance graph of a Cayley graph $\Cay(G,S)$ of $G$ equals $\Cay(G,S^n\setminus S)$ and hence it is itself a Cayley graph. Any isomorphism between $\Cay(G,S)$ and $\Cay(G,S^n\setminus S)$ give the constraint $|S^n|<2|S|$ on $S$, the problem which is the subject of resent research. On the other hand, such an isomorphism brings us the question whether $S^n\setminus S$ and $S$ are conjugate via an automorphism of $G$, which is a central problem in the theory of Cayley graphs. In case $S^n\setminus S=S^\theta$ for some $\theta\in\Aut(G)$, we have obviously $\Cay(G,S^n\setminus S)\cong\Cay(G,S)$, that is, $\Cay(G,S)$ is a self $n$th distance graph.

Throughout this paper, we use the following notations: The maximum degree of vertices of a graph $\G$ is denote by $\Delta(\G)$ and $N_{\G}(v)$ illustrates the set of all neighborhoods of the vertex $v$ in $\G$. Also, $\nabla(\G)$ denotes the number of triangles in a graph $\G$. All graphs in this papers are finite simple graphs with no multiple edges. Remind that a \textit{diamond} is the edge product $\D=C_3|C_3$, where the \textit{edged product} of two edge-transitive graphs $\G_1$ and $\G_2$ is obtained by identification of an edge from $\G_1$ and $\G_2$.
\section{Preliminary results}
We begin with a simple query about the existence of self $2$-distance graphs. Clearly, any odd cycle of length $\geq5$ is a self $2$-distance graph. As we shall see later, odd cycles are exceptional examples in the class of all self $2$-distance graphs. We note that the class of self $2$-distance graphs is broad as Propositions \ref{selfcomplementary} and \ref{inducedsubgraph} provide ample of them. The following simple key lemma plays an important role in our study.
\begin{lemma}\label{diameter}
Let $\G$ be a graph. Then $\diam(\G)=2$ if and only if $\G_2=\G^c$.
\end{lemma}
\begin{proposition}\label{selfcomplementary}
Let $\G$ be a self-complementary graph with diameter two. Then $\G_2\cong\G$.
\end{proposition}
\begin{proposition}\label{inducedsubgraph}
Every graph is an induced subgraph of a self $2$-distance graph.
\end{proposition}
\begin{proof}
Let $\G$ be an arbitrary graph. Consider two disjoint copies $\G_1$ and $\G_2$ of $\G$ and two disjoint copies $\G_3$ and $\G_4$ of $\G^c$, and let $v$ be a new vertex. Then the graph with vertex set
\[V(\G_1)\cup V(\G_2)\cup V(\G_3)\cup V(\G_4)\cup\{v\}\]
and edge set
\[E(\G_1)\cup E(\G_2)\cup E(\G_3)\cup E(\G_4)\cup E,\]
where
\[E=\{\{v,v_1\},\{v,v_2\},\{v_1,v_3\},\{v_2,v_4\},\{v_3,v_4\}:v_i\in V(\G_i),i=1,2,3,4\}\]
is a self $2$-distance graph containing $\G$ as a subgraph (see Figure 2.3.1).
\end{proof}

\begin{center}
\begin{tikzpicture}
\node [draw,shape=circle] (A) at ({cos(90+72*0)},{sin(90+72*0)}) {$v$};
\node [draw,shape=circle] (B) at ({cos(90+72*1)},{sin(90+72*1)}) {$\G_1$};
\node [draw,shape=circle] (C) at ({cos(90+72*2)},{sin(90+72*2)}) {$\G_3$};
\node [draw,shape=circle] (D) at ({cos(90+72*3)},{sin(90+72*3)}) {$\G_4$};
\node [draw,shape=circle] (E) at ({cos(90+72*4)},{sin(90+72*4)}) {$\G_2$};
\draw (A)--(B)--(C)--(D)--(E)--(A);
\end{tikzpicture}\\
Figure 2.3.1
\end{center}
\begin{lemma}
If $\G$ is a self $2$-distance graph which is not an odd cycle, then $\gr(\G)=3$.
\end{lemma}
\begin{proof}
Since $\Delta(\G)>2$, we may choose a vertex $v$ of degree $\geq3$. If $N_\G(v)$ is not empty, then $\G$ has a triangle. Thus we may assume that $N_\G(v)$ is empty. But then $N_\G(v)^c$ is a subgraph of $\G_2\cong\G$, which implies that $\G$ has a triangle. Therefore $\gr(\G)=3$.
\end{proof}

The following lemma will be used in the next section.
\begin{lemma}\label{linegraphdistancegraph}
Let $\G$ be a graph. Then 
\[|E(L(\G))|=|E(\G_2)|+|E(\G)|+3\nabla(\G)-\binom{|V(\G)|}{2}+\sum_{u\not\sim v}|N_\G(u)\cap N_\G(v)|.\]
In particular, 
\[|E(L(\G))|=|E(\G_2)|+3\nabla(\G)\]
whenever $\G$ has no $4$-cycle subgraph.
\end{lemma}
\begin{proof}
It is straightforward.
\end{proof}
\section{Graphs with no $4$-cycle subgraph}
Throughout this section, we assume that $\G\cong\G_2$ is a graph with no $4$-cycle as subgraph. Further, we assume that $\G$ is not an odd cycle. A simple observation shows that every triangle in $\G_2$ comes from an induced claw, an induces $6$-cycle or an induced edge product $C_5|C_3$. Moreover, every $6$-cycle in $\G$ is induced or it induces a graph isomorphic to $C_5|C_3$. To achieve the classification of graphs $\G$ with the mentioned properties, we need to analyze the existence of some special subgraphs of $\G$ as presented in Lemma \ref{C4-free:C5C3}--\ref{C4-free:C3C3}. The following lemma will be used in the sequel without reference.
\begin{lemma}\label{C4-free:Delta=3}
$\Delta(\G)=3$.
\end{lemma}
\begin{proof}
Since neither $\G$ nor $\G_2$ have $4$-cycles and $N_\G(v)^c$ is a subgraph of $\G_2$ for all $v\in V(\G)$, it follows that $\Delta(\G)\leq3$. Now, the fact that $\G$ is not a cycle, implies that $\Delta(\G)\geq3$ so that $\Delta(\G)=3$.
\end{proof}
\begin{lemma}\label{C4-free:C5C3}
If $\G$ has a $C_5|C_3$ subgraph, then $\G$ is isomorphic to $C_5|C_3$.
\end{lemma}
\begin{proof}
Suppose on the contrary that $\G\ncong C_5|C_3$ and $S\subset V(\G)$ induces a subgraph of $\G$ isomorphic to $C_5|C_3$ (see Figure 3.1.1). Then there exists a vertex $v\in V(\G)$ adjacent to some vertex of $S$. Clearly, $v$ is not adjacent to the temples for $\Delta(\G)=3$.

First suppose that $v$ is adjacent to the forehead. If $v$ is adjacent to any of the jaws, then we get a $4$-cycle, which is a contradiction. Thus $N_S(v)=\{a\}$ or $\{a,d\}$, which imply that $\{v,b,d,f\}$ is a $4$-cycle in $\G_2$, which is again a contradiction. Therefore $v$ is not adjacent to the forehead. Next assume that $v$ is adjacent to the chin. Clearly, $v$ is not adjacent to both $c$ and $d$, say $c$, for otherwise we have a $r$-cycle $\{c,d,e,v\}$. Bu then $\{a,f,e,v\}\subseteq N_{\G_2}(c)$, that is, $\Delta(\G_2)>3$, which is a contradiction. Finally, assume that $v$ is adjacent to any of the jaws. Then $v$ is adjacent to exactly one of the jaws, say $c$, for otherwise $\{v,c,d,e\}$ is a $4$-cycle. Since $(S\cup\{v\})_2\ncong S\cup\{v\}$, there exists yet another vertex $u\in V(\G)\setminus S\cup\{v\}$ adjacent to some vertex of $S\cup\{v\}$. If $u$ is adjacent to $v$, then either $N_{\G_2}(c)$ contains $\{a,e,f,u\}$ as $u$ is cannot be adjacent to $c$, which is a contradictions. Thus $u$ is not adjacent to $v$ and by the same arguments as before $u$ is adjacent to one of the jaws. Since $u$ and $c$ are not adjacent, $u$ and $e$ must be adjacent, which implies that $\{b,f,u,v\}\subseteq N_{\G_2}(d)$, a contradiction. The proof is complete.
\end{proof}
\begin{center}
\begin{tikzpicture}[scale=1]
\node [circle,draw,inner sep=1pt,label=above:{forehead}] (A) at (90:1cm) {a};
\node [circle,draw,inner sep=1pt,label=left:{left temple}] (B) at (150:1cm) {b};
\node [circle,draw,inner sep=1pt,label=left:{left jaw}] (C) at (210:1cm) {c};
\node [circle,draw,inner sep=1pt,label=below:{chin}] (D) at (270:1cm) {d};
\node [circle,draw,inner sep=1pt,label=right:{right jaw}] (E) at (330:1cm) {e};
\node [circle,draw,inner sep=1pt,label=right:{right temple}] (F) at (30:1cm) {f};
\draw (A)--(B)--(C)--(D)--(E)--(F)--(A); 
\draw (B)--(F); 
\end{tikzpicture}\\
Figure 3.1.1
\end{center}
\begin{lemma}\label{C4-free:C5}
 If $\G$ has a $5$-cycle, then $\G$ is isomorphic to $C_5|C_3$.
\end{lemma}
\begin{proof}
Since $\G\ncong C_5$, then there exist a vertex $v\in V(\G)\setminus S$ adjacent to some vertex $u$ of $S$, where $S$ is a $5$-cycle in $\G$. Clearly, $S$ is an induced subgraph of $\G$. Let $a,b$ be two vertices adjacent to $u$ in $S$ and $c,d$ be two other vertices. Since $\G$ has no $4$-cycle it follows that $v$ is not adjacent to $c,d$. Now, it is easy to see that either $\G$ or $\G_2$ has a subgraph isomorphic to $C_5|C_3$, from which by Lemma \ref{C4-free:C5C3}, it follows that $\G\cong C_5|C_3$.
\end{proof}
\begin{lemma}\label{C4-free:C6}
If $\G$ has a $6$-cycle, then $\G$ is isomorphic to $C_5|C_3$.
\end{lemma}
\begin{proof}
If $\G$ has a $C_5|C_3$ subgraph, then we are done. Thus we may assume that $\G$ has no subgraph isomorphic to $C_5|C_3$. Let $S\subset V(\G)$ denote a $6$-cycle $a,b,c,d,e,f,a$ in $\G$. Clearly, $S$ is an induced subgraph of $\G$. Since $S_2\ncong S$, we have a vertex $u\in V(\G)\setminus S$ adjacent to some vertex $a$ of $S$. Clearly, $u$ is adjacent to exactly one of $b,f$, say $b$, for otherwise either $\{b,d,f,u\}$ is a $4$-cycle in $\G_2$, or $\{b,a,f,u\}$ is a $4$-cycle in $\G$, which are both impossible. Again, the fact that $\G$ has no $4$-cycle implies that $u$ is not adjacent to $c,d,e,f$. Moreover, $u$ is the unique vertex adjacent to both $a,b$. Now, we have three cases:

Case 1. If $\G$ has a subgraph $T$ as drawn in Figure 3.3.3, then $T$ is an induced subgraph and a simple verification shows that $T$ is a connected component of $\G$, which implies that $\G=T$. But then $\G_2\ncong\G$, which is a contradiction.

Case 2. If $\G$ has a subgraph $T$ as drawn in Figure 3.3.2, then since $T_2\ncong T$, $\G$ has a vertex $w'$ adjacent to some vertex of $T$. If $w'$ is adjacent to any of the vertices $a',b',c',d',u',v'$, then we get a vertex of degree $\geq4$ in $\G$ or $\G_2$, which is impossible. Thus $w'$ is adjacent to $e'$ or $f'$ and by the previous argument it follows that $w'$ is adjacent to both $e'$ and $f'$, which is impossible by Case 1.

Case 3. $\G$ has no subgraphs isomorphic to that of Figure 3.3.2. Then $u$ is the only vertex of $\G$ adjacent to $S$ (see Figure 3.3.1). Since $(S\cup\{u\})_2\ncong S\cup\{u\}$, there exists a vertex $v\in V(\G)\setminus S\cup\{u\}$ adjacent to $u$. But then $(S\cup\{u,v\})_2$ is an induced subgraph of $\G_2\cong\G$ isomorphic to the graph in Figure 3.3.4, from which it follows that $\deg_{(\G_2)_2}(u)\geq4$, a contradiction.
\end{proof}
\begin{center}
\begin{tabular}{ccccc}
\begin{tikzpicture}[scale=1]
\node at (0,-2.07) {};
\node [circle,fill=black,inner sep=1pt] (A) at (0:1cm) {};
\node [circle,fill=black,inner sep=1pt] (B) at (60:1cm) {};
\node [circle,fill=black,inner sep=1pt] (C) at (120:1cm) {};
\node [circle,fill=black,inner sep=1pt] (D) at (180:1cm) {};
\node [circle,fill=black,inner sep=1pt] (E) at (240:1cm) {};
\node [circle,fill=black,inner sep=1pt] (F) at (300:1cm) {};
\node [circle,fill=black,inner sep=1pt] (G) at (30:1.732cm) {};
\draw (A)--(B)--(C)--(D)--(E)--(F)--(A);
\draw (A)--(G)--(B);
\node at (0:1.25cm) {$a$};
\node at (60:1.25cm) {$b$};
\node at (120:1.25cm) {$c$};
\node at (180:1.25cm) {$d$};
\node at (240:1.25cm) {$e$};
\node at (300:1.25cm) {$f$};
\node at (30:1.982cm) {$u$};
\end{tikzpicture}
&&
\begin{tikzpicture}[scale=1]
\node at (0,-2.07) {};
\node [circle,fill=black,inner sep=1pt] (A) at (0:1cm) {};
\node [circle,fill=black,inner sep=1pt] (B) at (60:1cm) {};
\node [circle,fill=black,inner sep=1pt] (C) at (120:1cm) {};
\node [circle,fill=black,inner sep=1pt] (D) at (180:1cm) {};
\node [circle,fill=black,inner sep=1pt] (E) at (240:1cm) {};
\node [circle,fill=black,inner sep=1pt] (F) at (300:1cm) {};
\node [circle,fill=black,inner sep=1pt] (G) at (30:1.732cm) {};
\node [circle,fill=black,inner sep=1pt] (H) at (150:1.732cm) {};
\draw (A)--(B)--(C)--(D)--(E)--(F)--(A);
\draw (A)--(G)--(B);
\draw (C)--(H)--(D);
\node at (0:1.25cm) {$a'$};
\node at (60:1.25cm) {$b'$};
\node at (120:1.25cm) {$c'$};
\node at (180:1.25cm) {$d'$};
\node at (240:1.25cm) {$e'$};
\node at (300:1.25cm) {$f'$};
\node at (30:1.982cm) {$u'$};
\node at (150:1.982cm) {$v'$};
\end{tikzpicture}
&&
\begin{tikzpicture}[scale=1]
\node [circle,fill=black,inner sep=1pt] (A) at (0:1cm) {};
\node [circle,fill=black,inner sep=1pt] (B) at (60:1cm) {};
\node [circle,fill=black,inner sep=1pt] (C) at (120:1cm) {};
\node [circle,fill=black,inner sep=1pt] (D) at (180:1cm) {};
\node [circle,fill=black,inner sep=1pt] (E) at (240:1cm) {};
\node [circle,fill=black,inner sep=1pt] (F) at (300:1cm) {};
\node [circle,fill=black,inner sep=1pt] (G) at (30:1.732cm) {};
\node [circle,fill=black,inner sep=1pt] (H) at (150:1.732cm) {};
\node [circle,fill=black,inner sep=1pt] (I) at (270:1.732cm) {};	
\draw (A)--(B)--(C)--(D)--(E)--(F)--(A);
\draw (A)--(G)--(B);
\draw (C)--(H)--(D);
\draw (E)--(I)--(F);
\node at (0:1.25cm) {$a'$};
\node at (60:1.25cm) {$b'$};
\node at (120:1.25cm) {$c'$};
\node at (180:1.25cm) {$d'$};
\node at (240:1.25cm) {$e'$};
\node at (300:1.25cm) {$f'$};
\node at (30:1.982cm) {$u'$};
\node at (150:1.982cm) {$v'$};
\node at (270:1.982cm) {$w'$};
\end{tikzpicture}\\
\mbox{Figure 3.3.1}&&\mbox{Figure 3.3.2}&&\mbox{Figure 3.3.3}
\end{tabular}

\begin{tabular}{ccccc}
&&
\begin{tikzpicture}[scale=1]
\node [circle,fill=black,inner sep=1pt] (A) at (0:1cm) {};
\node [circle,fill=black,inner sep=1pt] (B) at (60:1cm) {};
\node [circle,fill=black,inner sep=1pt] (C) at (120:1cm) {};
\node [circle,fill=black,inner sep=1pt] (D) at (180:1cm) {};
\node [circle,fill=black,inner sep=1pt] (E) at (240:1cm) {};
\node [circle,fill=black,inner sep=1pt] (F) at (300:1cm) {};
\node [circle,fill=black,inner sep=1pt] (G) at (30:1.732cm) {};
\node [circle,fill=black,inner sep=1pt] (H) at (210:1.732cm) {};
\draw (A)--(B)--(C)--(D)--(E)--(F)--(A);
\draw (A)--(G)--(B);
\draw (D)--(H)--(E);
\node at (0:1.25cm) {$a$};
\node at (60:1.25cm) {$c$};
\node at (120:1.25cm) {$u$};
\node at (180:1.25cm) {$f$};
\node at (240:1.25cm) {$b$};
\node at (300:1.25cm) {$v$};
\node at (30:1.982cm) {$e$};
\node at (210:1.982cm) {$d$};
\end{tikzpicture}
&&
\\
&&\mbox{Figure 3.3.4}&&
\end{tabular}
\end{center}
\begin{lemma}\label{C4-free:Cn}
If $\G$ is not isomorphic to $C_5|C_3$, then $\G$ has no cycles of length exceeding three.
\end{lemma}
\begin{proof}
By Lemmas \ref{C4-free:C5} and \ref{C4-free:C6} and hypothesis, $\G$ has no cycles of lengths $4,5$ or $6$. We proceed by induction to show that $\G$ has no cycles of lengths $\geq4$. Suppose $\G$ has no cycles of lengths $4,5,\ldots,n$ for some $n\geq6$. If $\G$ has an $(n+1)$-cycle $C$, then $C$ is an induced subgraph of $\G$. If $n+1$ is even, then clearly $\G_2$ has two $(n+1)/2$-cycles, which is a contradiction. Thus $n+1$ is odd. Since $\G$ is not an odd cycle, there exists a vertex $v\in V(\G)$ adjacent to some vertex $a\in V(C)$. Let $N_C(a)=\{b,c\}$. If $v$ is adjacent to some vertex in $C\setminus\{a,b,c\}$, then we obtain a cycle of length $l$ ($4\leq l\leq n$), which is a contradiction. If $v$ is not adjacent to $b,c$, then $\G\cong\G_2$ has a subgraph isomorphic to $(C\cup\{v\})_2$ that is an $|C|$-cycle with two adjacent vertices having a common neighbor. Hence, we may assume that $v$ is adjacent $b$ or $c$, say $b$. Since $\G$ has no $4$-cycle, $v$ is not adjacent to $c$. Let $N_C(b)=\{a,d\}$. Then $c,v,d$ is a path of length two in $\G_2$.
On the other hand, since $C_2$ is a subgraph of $\G_2$, there is a path of length at most $n/2$ from $c$ to $d$ disjoint from $c,v,d$. Hence $\G_2$ has a cycle of length $l$ such that $4\leq l\leq n/2+2\leq n$, which is a contradiction. The proof is complete.
\end{proof}
\begin{lemma}\label{C4-free:C3C3}
Triangles in $\G$ have disjoint vertices.
\end{lemma}
\begin{proof}
If two triangles of $\G$ have some vertices in common, then either $\G$ or $\G_2$ has a $4$-cycle, which is a contradiction.
\end{proof}

Now, we are ready to prove the main result of this section. To end this, we use the notion of distance between two subgraphs of a graph as the length of the shortest path connecting a vertex of the first subgraph to a vertex of the second subgraph.
\begin{theorem}\label{C4free}
Let $\G$ be a self $2$-distance graph with no $4$-cycle. Then either $\G$ is an odd cycle or it is the edged product $C_5|C_3$.
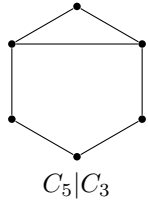
\begin{center}
\begin{tikzpicture}[scale=1]
\node [circle,fill=black,inner sep=1pt] (A) at (30:1cm) {};
\node [circle,fill=black,inner sep=1pt] (B) at (90:1cm) {};
\node [circle,fill=black,inner sep=1pt] (C) at (150:1cm) {};
\node [circle,fill=black,inner sep=1pt] (D) at (210:1cm) {};
\node [circle,fill=black,inner sep=1pt] (E) at (270:1cm) {};
\node [circle,fill=black,inner sep=1pt] (F) at (330:1cm) {};
\draw (A)--(B)--(C)--(D)--(E)--(F)--(A)--(C); 
\end{tikzpicture}\\
$C_5|C_3$
\end{center}
\end{theorem}

\begin{proof}
Let $\G'$ be the graph obtained from $\G$ by contracting all triangles into single vertices. By Lemmas \ref{C4-free:Cn} and \ref{C4-free:C3C3}, $\G'$ is a tree. Let $v$ and $v'$ ($e$ and $e'$) be the number of vertices (edges) of $\G$ and $\G'$, respectively. Also, let $n_i$ be the number of vertices of degree $i$ in $\G$ for $i=1,2,3$. Clearly, $v'=v-2\nabla(\G)$ and $e'=e-3\nabla(\G)$. Since $\G'$ is a tree, we have $e'=v'-1$, which implies that $\nabla(\G)=e-v+1$. On the other hand, by Lemma \ref{linegraphdistancegraph}, $e_L-e=3\nabla(\G)$, where $e_L$ is the number of edges of $L(\G)$, the line graph of $\G$. Now, we have
\begin{align*}
|V(\G)|&=n_1+n_2+n_3,\\
|E(\G)|&=\frac{1}{2}\sum_{v\in V(\G)}\deg_\G(v)=\frac{n_1+2n_2+3n_3}{2},\\
|E(L(\G))|&=\sum_{v\in V}\binom{\deg_\G(v)}{2}=n_2+3n_3,
\end{align*}
from which it follows that $n_1=3$.

If $\G$ has no triangles then $\G$ is a tree so that $\G_2$ is disconnected, which is a contradiction. Hence $\G$ has some triangles. A triangle in $\G$ is said to be $i$-tailed if it contains $i$ cubic vertices. Clearly, $\G$ has no $3$-tailed triangle for otherwise $\G_2$ must have a hexagon contradicting Lemma \ref{C4-free:C6}. Suppose $\G$ has no $1$-tailed triangle. Hence, we have no induced claws with two pendants, which implies that $\G$ has only one induced claw along with only one $2$-tailed triangle as drawn in Figure 3.7.1, where $a,b,d\geq1$ and $c\geq0$. Clearly, $c\neq1$ for otherwise $\deg_{\G_2}(u)=4$, which is impossible. A simple verification shows that $d_\G(\mbox{triangle},\mbox{claw})=c$ and 
\[d_{\G_2}(\mbox{triangle},\mbox{claw})=\begin{cases}\frac{c+4}{2},&c\ \mbox{is even},\\\\\frac{c-3}{2},&c\ \mbox{is odd}.\end{cases}\]
Since $\G\cong\G_2$ this implies that $c=4$. On the other hand, we know that 
\[|E(\G)|=a+b+c+d+4\]
and
\[|E(\G_2)|=a+b+c+d+5-[\frac{1}{d}]\]
when $c\geq2$. But then $d=1$ and $a\pm1,b\mp1=2,3$, from which it follows that $\G_2\not\cong\G$, a contradiction.

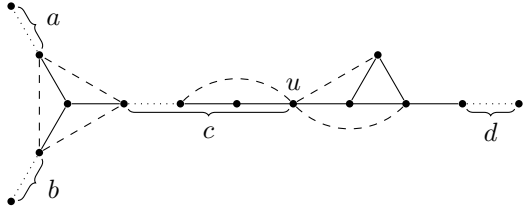
\begin{center}
\begin{tikzpicture}[scale=0.75]
\node [circle,fill=black,inner sep=1pt] (O) at (0,0) {};
\node [circle,fill=black,inner sep=1pt] (C1) at (0:1cm) {};
\node [circle,fill=black,inner sep=1pt] (C2) at (0:2cm) {};
\node [circle,fill=black,inner sep=1pt] (C3) at (0:3cm) {};
\node [circle,fill=black,inner sep=1pt,label=above:{$u$}] (C4) at (0:4cm) {};
\node [circle,fill=black,inner sep=1pt] (C5) at (0:5cm) {};
\node [circle,fill=black,inner sep=1pt] (C6) at (0:6cm) {};
\node [circle,fill=black,inner sep=1pt] (C7) at (0:7cm) {};
\node [circle,fill=black,inner sep=1pt] (C8) at (0:8cm) {};

\node [circle,fill=black,inner sep=1pt] (C10) at ({5+cos(60)},{sin(60)}) {};
\node [circle,fill=black,inner sep=1pt] (B1) at (120:1cm) {};
\node [circle,fill=black,inner sep=1pt] (B2) at (120:2cm) {};
\node [circle,fill=black,inner sep=1pt] (A1) at (240:1cm) {};
\node [circle,fill=black,inner sep=1pt] (A2) at (240:2cm) {};

\draw [dotted] (C1)--(C2);
\draw [dotted] (C7)--(C8);
\draw [dotted] (B1)--(B2);
\draw [dotted] (A1)--(A2);
\draw (O)--(C1);
\draw (O)--(B1);
\draw (O)--(A1);
\draw (C2)--(C7);

\draw (C5)--(C10)--(C6);

\draw [decorate,decoration={brace,mirror,raise=2pt}] (C1)--(C4);
\draw [decorate,decoration={brace,mirror,raise=2pt}] (C7)--(C8);
\draw [decorate,decoration={brace,mirror,raise=2pt}] (B1)--(B2);
\draw [decorate,decoration={brace,raise=2pt}] (A1)--(A2);

\node at (-0.25,1.5) {$a$};
\node at (-0.25,-1.5) {$b$};
\node at (2.5,-0.5) {$c$};
\node at (7.5,-0.5) {$d$};

\draw [dashed] (C1)--(B1)--(A1)--(C1);
\draw [dashed] (C4) to [out=135, in=45] (C2);
\draw [dashed] (C4) to [out=-45, in=225] (C6);
\draw [dashed] (C4)--(C10);
\end{tikzpicture}\\
Figure 3.7.1
\end{center}

Therefore, $\G$ has a $1$-tailed triangle. Such a triangle arises from an induced claw with two pendants in $\G$. Since $\G$ has exactly three pendants, it can be drawn in the plane (see Figure 3.7.2) with one further triangle having an edge in the dotted areas, where $a,b\geq0$ and $c\geq1$ denote the number of vertices in the corresponding dotted areas. We note that every triangle in $\G_2$ arises from an induced claw in $\G$. A simple verification shows that
\[|E(\G)|=a+b+c+9\]
and
\[|E(\G_2)|=a+b+c+8+[\frac{1}{a+1}]+[\frac{1}{b+1}],\]
which implies that $ab=0$. Clearly, $c=1$ for otherwise $\deg_{\G_2}(o)\geq4$, which is impossible.

\begin{center}
\begin{tikzpicture}[scale=0.75]
\node [circle,fill=black,inner sep=1pt] (A3) at (0,0) {};
\node [circle,fill=black,inner sep=1pt] (A2) at (1,0) {};
\node [circle,fill=black,inner sep=1pt] (A1) at (2,0) {};
\node [circle,fill=black,inner sep=1pt] (O) at (3,0) {};
\node [circle,fill=black,inner sep=1pt] (B1) at (4,0) {};
\node [circle,fill=black,inner sep=1pt] (B2) at (5,0) {};
\node [circle,fill=black,inner sep=1pt] (B3) at (6,0) {};
\node [circle,fill=black,inner sep=1pt] (C1) at (3,1) {};
\node [circle,fill=black,inner sep=1pt] (C2) at (3,2) {};

\node [circle,fill=black,inner sep=1pt] (D1) at ({cos(150)},{sin(150)}) {};
\node [circle,fill=black,inner sep=1pt] (D2) at ({cos(-150)},{sin(-150)}) {};

\node [circle,fill=black,inner sep=1pt] (E1) at ({6+cos(30)},{sin(30)}) {};
\node [circle,fill=black,inner sep=1pt] (E2) at ({6+cos(-30)},{sin(-30)}) {};

\draw (D1)--(A3)--(D2);
\draw (A3)--(A2);
\draw (A1)--(B1);
\draw (B2)--(B3)--(E1)--(E2)--(B3);
\draw (O)--(C1);

\draw [dotted] (A2)--(A1);
\draw [dotted] (B1)--(B2);
\draw [dotted] (C1)--(C2);

\draw [decorate,decoration={brace,raise=2pt}] (A1)--(A2);
\draw [decorate,decoration={brace,raise=2pt}] (B2)--(B1);
\draw [decorate,decoration={brace,raise=2pt}] (C1)--(C2);

\node at (1.5,-0.5) {$a$};
\node at (4.5,-0.5) {$b$};
\node at (2.5,1.5) {$c$};
\node at (3,-0.5) {$o$};
\end{tikzpicture}\\
Figure 3.7.2
\end{center}

First assume that $a=0$. Then the graph $\G$ can be drawn  as in Figure 3.7.3. Note that $|A|\geq3$ for otherwise $A$ has a vertex of degree $\geq4$ in $\G_2$, which is impossible. This implies that two triangles in $\G_2$ are at distance at least five and so we must have $|B|\geq4$. But then we obtain three induced claws in $\G_2$ as drawn in Figure 3.7.3 with dashes, which is a contradiction.

\begin{center}
\begin{tikzpicture}[scale=0.75]
\node at (1.0,-0.5) {$o$};
\node at (2.5,0.75) {$A$};
\node at (9.5,0.75) {$B$};
\fill [gray!50] (2.5,0) ellipse (1.75cm and 0.3cm);
\fill [gray!50] (9.5,0) ellipse (2.75cm and 0.3cm);

\node [circle,fill=black,inner sep=1pt] (O) at (1,0) {};
\node [circle,fill=black,inner sep=1pt] (A1) at (2,0) {};
\node [circle,fill=black,inner sep=1pt] (A2) at (3,0) {};
\node [circle,fill=black,inner sep=1pt] (A3) at (4,0) {};

\node [circle,fill=black,inner sep=1pt] (T1) at (5,0) {};
\node [circle,fill=black,inner sep=1pt] (T2) at (6,0) {};
\node [circle,fill=black,inner sep=1pt] (T3) at ({5+cos(60)},{sin(60)}) {};

\node [circle,fill=black,inner sep=1pt] (B1) at (7,0) {};
\node [circle,fill=black,inner sep=1pt] (B2) at (8,0) {};
\node [circle,fill=black,inner sep=1pt] (B3) at (9,0) {};
\node [circle,fill=black,inner sep=1pt] (B4) at (10,0) {};
\node [circle,fill=black,inner sep=1pt] (B5) at (11,0) {};
\node [circle,fill=black,inner sep=1pt] (B6) at (12,0) {};

\node [circle,fill=black,inner sep=1pt] (C1) at (1,1) {};

\node [circle,fill=black,inner sep=1pt] (D0) at (0,0) {};
\node [circle,fill=black,inner sep=1pt] (D1) at ({cos(150)},{sin(150)}) {};
\node [circle,fill=black,inner sep=1pt] (D2) at ({cos(-150)},{sin(-150)}) {};

\node [circle,fill=black,inner sep=1pt] (E0) at (13,0) {};
\node [circle,fill=black,inner sep=1pt] (E1) at ({13+cos(30)},{sin(30)}) {};
\node [circle,fill=black,inner sep=1pt] (E2) at ({13+cos(-30)},{sin(-30)}) {};

\draw (D1)--(D0)--(D2);
\draw (D0)--(O);
\draw (O)--(C1);

\draw (A1)--(B3);
\draw (B4)--(E0)--(E1)--(E2)--(E0);
\draw (T1)--(T3)--(T2);
\draw [dotted] (O)--(A1);
\draw [dotted] (B3)--(B4);

\draw [dashed] (A3)--(T3);
\draw [dashed] (A1) to [out=-45, in=-135] (A3);
\draw [dashed] (A3) to [out=-45, in=-135] (T2);

\draw [dashed] (B1)--(T3);
\draw [dashed] (T1) to [out=-45, in=-135] (B1);
\draw [dashed] (B1) to [out=-45, in=-135] (B3);

\draw [dashed] (E1)--(B6)--(E2);
\draw [dashed] (B4) to [out=-45, in=-135] (B6);
\end{tikzpicture}\\
Figure 3.7.3
\end{center}

Finally assume that $b=0$. If $|A|=1$, then two induced claws are connected with two triangles with distance zero while it is not true in $\G_2$. Hence, $|A|\geq2$. If $|A|=2$, then $A$ has a vertex of degree four in $\G_2$, which is impossible. Thus $|A|\geq3$. Similarly, $|B|\geq3$ for otherwise it has a vertex of degree four in $\G_2$, a contradiction. But then we obtain three induced claws in $\G_2$ as drawn in Figure 3.7.4 with dashes, which is a contradiction.

\begin{center}
\begin{tikzpicture}[scale=0.75]
\node at (11.0,-0.5) {$o$};
\node at (1.5,0.75) {$A$};
\node at (8.5,0.75) {$B$};
\fill [gray!50] (1.5,0) ellipse (1.75cm and 0.3cm);
\fill [gray!50] (8.5,0) ellipse (2.75cm and 0.3cm);

\node [circle,fill=black,inner sep=1pt] (A1) at (1,0) {};
\node [circle,fill=black,inner sep=1pt] (A2) at (2,0) {};
\node [circle,fill=black,inner sep=1pt] (A3) at (3,0) {};

\node [circle,fill=black,inner sep=1pt] (T1) at (4,0) {};
\node [circle,fill=black,inner sep=1pt] (T2) at (5,0) {};
\node [circle,fill=black,inner sep=1pt] (T3) at ({4+cos(60)},{sin(60)}) {};

\node [circle,fill=black,inner sep=1pt] (B1) at (6,0) {};
\node [circle,fill=black,inner sep=1pt] (B2) at (7,0) {};
\node [circle,fill=black,inner sep=1pt] (B3) at (8,0) {};
\node [circle,fill=black,inner sep=1pt] (B4) at (9,0) {};
\node [circle,fill=black,inner sep=1pt] (B5) at (10,0) {};

\node [circle,fill=black,inner sep=1pt] (C1) at (11,1) {};

\node [circle,fill=black,inner sep=1pt] (D0) at (0,0) {};
\node [circle,fill=black,inner sep=1pt] (D1) at ({cos(150)},{sin(150)}) {};
\node [circle,fill=black,inner sep=1pt] (D2) at ({cos(-150)},{sin(-150)}) {};

\node [circle,fill=black,inner sep=1pt] (E0) at (12,0) {};
\node [circle,fill=black,inner sep=1pt] (E1) at ({12+cos(30)},{sin(30)}) {};
\node [circle,fill=black,inner sep=1pt] (E2) at ({12+cos(-30)},{sin(-30)}) {};

\node [circle,fill=black,inner sep=1pt] (O) at (11,0) {};

\draw (D1)--(D0)--(D2);
\draw (A1)--(B3);
\draw (B4)--(E0)--(E1)--(E2)--(E0);
\draw (O)--(C1);
\draw (T1)--(T3)--(T2);
\draw [dotted] (D0)--(A1);
\draw [dotted] (B3)--(B4);

\draw [dashed] (A3)--(T3);
\draw [dashed] (A1) to [out=-45, in=-135] (A3);
\draw [dashed] (A3) to [out=-45, in=-135] (T2);

\draw [dashed] (B1)--(T3);
\draw [dashed] (T1) to [out=-45, in=-135] (B1);
\draw [dashed] (B1) to [out=-45, in=-135] (B3);

\draw [dashed] (E1)--(O)--(E2);
\draw [dashed] (B4) to [out=-45, in=-135] (O);
\end{tikzpicture}\\
Figure 3.7.4
\end{center}
The proof is complete.
\end{proof}
\section{Graphs with disjoint triangles}
Throughout this section, we assume that $\G\cong\G_2$ is a graph with disjoint triangles. Further we assume that $\G$ is not an odd cycle. As in section 3, we proceed by analysing the existence of special subgraphs in $\G$ is several lemmas. The following lemma is crucial in the proof of our results.
\begin{lemma}\label{disjoint-triangles:Delta}
We have $\Delta(\G)=3$.
\end{lemma}
\begin{proof}
Let $v$ be a vertex of $\G$. Clearly, $N_\G(v)$ is a union of isolated vertices and at most one edge. Now, since $N_\G(v)^c$ is a subgraph of $\G_2$, we must have $|N_\G(v)|\leq3$, as required.
\end{proof}
\begin{lemma}\label{disjoint-triangles:C5|C3}
If $\G$ has a $C_5|C_3$ subgraph, then $\G$ is isomorphic to $C_5|C_3$.
\end{lemma}
\begin{proof}
Suppose on the contrary that $\G$ is not isomorphic to $C_5|C_3$ and consider a subgraph $S$ of $\G$ isomorphic to $C_5|C_3$ as drawn in Figure 3.1.1. We proceed in two steps.

Case 1. The jaws are non-adjacent. Hence there is a vertex in $\G\setminus S$ adjacent to some vertex of $S$. First suppose that the chin $d$ is adjacent to some new vertex $g$. If $g$ is not adjacent to jaws $c$ and $e$, then we have two triangles $\{a,c,e\}$ and $\{c,e,g\}$ with a common edge in $\G_2$ contradicting the assumption. Hence, $g$ is adjacent to exactly one of the jaws, say $c$. But then we have two triangles $\{a,c,e\}$ and $\{b,e,g\}$ in $\G_2$ with a common vertex, which is another contradiction. Therefore, $N_\G(d)=\{c,e\}$. Next assume that a jaw, say $c$, is adjacent to a new vertex $g$. Clearly, $b,d,f\notin N_\G(g)$. If $g$ and $e$ are adjacent, then we have two triangles $\{b,d,g\}$ and $\{d,f,g\}$ with a common edge in $\G_2$, a contradiction. Hence $N_S(g)=\{c\}$ and the subgraph induce by $\{a,b,c,d,e,f\}$ in $\G_2$ is isomorphic to $C_5|C_3$ with $g$ adjacent to its chin, which is impossible by the previous discussions. Clearly the temples are not adjacent to any vertex of $\G\setminus S$. Hence, the forehead $a$ must be adjacent to a new vertex $g$ so that the subgraph induced by $\{a,b,c,d,e,f\}$ in $\G_2$ is isomorphic to $C_5|C_3$ with $g$ adjacent to its jaws, which is impossible by previous arguments.

Case 2. The jaws are adjacent. Since the subgraph induced by $S$ is not a self $2$-distance graph, one of the foreheads or chin must be adjacent to a new vertex $g$, say $a$ and $g$ are adjacent. Then we have two triangles $\{c,d,e\}$ and $\{c,e,g\}$ in $(\G_2)_2$, which is a contradiction.
\end{proof}
\begin{lemma}\label{disjoint-triangles:C6}
The graph $\G$ does not have any hexagon.
\end{lemma}
\begin{proof}
Suppose on the contrary that $\G$ has a hexagon $S$ as in Figure 4.3.1 with vertices $a,b,c,d,e,f$. Since there is no subgraph isomorphic to $C_5|C_3$ in $\G$, the only possible chords of $S$ are $\{a,d\}$, $\{b,e\}$ or $\{c,f\}$. Since $S$ is not a self $2$-distance graph, we may assume that $a$ is adjacent to a new vertex $g$. Clearly, $g$ is adjacent to exactly one of $b$ or $f$, say $b$, for otherwise either $\G$ or $\G_2$ has two triangles with a common edge. Now, by using Lemma \ref{disjoint-triangles:C5|C3}, one can easily see that the vertices $a,b,c,d,e,f,g$ induce a subgraph in $\G_2$ as drawn with dashes in Figure 4.3.1. Hence, the degree of $g$ in $(\G_2)_2$ is at least four, which is a contradiction.
\end{proof}
\begin{center}
\begin{tikzpicture}[scale=1]
\node [circle,fill=black,inner sep=1pt] (A) at (60:1cm) {};
\node [circle,fill=black,inner sep=1pt] (B) at (120:1cm) {};
\node [circle,fill=black,inner sep=1pt] (C) at (180:1cm) {};
\node [circle,fill=black,inner sep=1pt] (D) at (240:1cm) {};
\node [circle,fill=black,inner sep=1pt] (E) at (300:1cm) {};
\node [circle,fill=black,inner sep=1pt] (F) at (360:1cm) {};
\node [circle,fill=black,inner sep=1pt] (G) at (90:1.732cm) {};
\draw (A)--(B)--(C)--(D)--(E)--(F)--(A)--(G)--(B); 
\draw [dashed] (A)--(C)--(E)--(A);
\draw [dashed] (B)--(D)--(F)--(B);
\draw [dashed] (G) to [out=200,in=100] (C);
\draw [dashed] (G) to [out=-20,in=80] (F);
\node  at (60:1.25cm) {$a$};
\node  at (120:1.25cm) {$b$};
\node  at (180:1.25cm) {$c$};
\node  at (240:1.25cm) {$d$};
\node  at (300:1.25cm) {$e$};
\node  at (360:1.25cm) {$f$};
\node  at (90:2cm) {$g$};
\end{tikzpicture}\\
Figure 4.3.1
\end{center}
\begin{lemma}\label{disjoint-triangles:C5}
The graph $\G$ does not have any pentagon.
\end{lemma}
\begin{proof}
Suppose on the contrary that $\G$ has a pentagon $S$ with vertices $a,b,c,d,e$. We consider two cases:

Case 1. $S$ does not have any chord. Since $\G$ is not an odd cycle, we may assume that $a$ is adjacent to a new vertex $f$. By Lemma \ref{disjoint-triangles:C5|C3}, $f$ is not adjacent to $b$ and $e$, from which it follows that $\G_2$ has a subgraph isomorphic to $C_5|C_3$, a contradiction.

Case 2. $S$ has a chord. Clearly, $S$ has a unique chord, say $\{b,e\}$. Since $S$ is not a self $2$-distance graph, it has a vertex adjacent to a new vertex $f$. First suppose that $a$ and $f$ are adjacent. Since $\G_2$ does not have a subgraph isomorphic to $C_5|C_3$, either $c,d\in N_\G(f)$ or $c,d\notin N_\G(f)$. In both cases, the vertices $a,b,c,d,e,f$ induce a hexagon in $\G_2$, contradicting Lemma \ref{disjoint-triangles:C6} (see Figure 4.4.1). Therefore, $f$ is adjacent to $c$ or $d$, say $c$. Clearly, $N_S(f)=\{c\}$. If there is a vertex $g$ adjacent to $d$, then again $N_S(g)=\{d\}$. Now, by using Lemma \ref{disjoint-triangles:C5|C3}, one can easily see that the vertices $a,b,c,d,e,f,g$ induce a subgraph in $\G_2$ as drawn with dashes in Figure 4.4.2. Hence, the degree of $a$ in $(\G_2)_2$ is at least four, which is a contradiction. Therefore, $d$ is not adjacent to vertices other than $c$ and $e$. This implies that the vertex $f$ is adjacent to another vertex $g$ as in Figure 4.4.3. Again, by using Lemma \ref{disjoint-triangles:C5|C3}, the vertices $a,b,c,d,e,f,g$ induce a subgraph in $\G_2$ as drawn with dashes in Figure 4.4.3. Hence, the degree of $a$ in $(\G_2)_2$ is at least four, which is a contradiction. The proof is complete.
\end{proof}
\begin{center}
\begin{tabular}{ccccc}
\begin{tikzpicture}[scale=1]
\node [circle,fill=black,inner sep=1pt] (A) at (90:1cm) {};
\node [circle,fill=black,inner sep=1pt] (B) at (162:1cm) {};
\node [circle,fill=black,inner sep=1pt] (C) at (234:1cm) {};
\node [circle,fill=black,inner sep=1pt] (D) at (306:1cm) {};
\node [circle,fill=black,inner sep=1pt] (E) at (378:1cm) {};
\node [circle,fill=black,inner sep=1pt] (F) at (90:2cm) {};
\draw (A)--(B)--(C)--(D)--(E)--(A)--(F);
\draw (B)--(E);
\draw [dashed] (A)--(C)--(E)--(F)--(B)--(D)--(A);
\node  at (80:1.25cm) {$a$};
\node  at (162:1.25cm) {$b$};
\node  at (234:1.25cm) {$c$};
\node  at (306:1.25cm) {$d$};
\node  at (378:1.25cm) {$e$};
\node  at (90:2.25cm) {$f$};
\end{tikzpicture}
&&
\begin{tikzpicture}[scale=1]
\node [circle,fill=black,inner sep=1pt] (A) at (90:1cm) {};
\node [circle,fill=black,inner sep=1pt] (B) at (162:1cm) {};
\node [circle,fill=black,inner sep=1pt] (C) at (234:1cm) {};
\node [circle,fill=black,inner sep=1pt] (D) at (306:1cm) {};
\node [circle,fill=black,inner sep=1pt] (E) at (378:1cm) {};
\node [circle,fill=black,inner sep=1pt] (F) at (234:2cm) {};
\node [circle,fill=black,inner sep=1pt] (G) at (306:2cm) {};
\draw (F)--(C)--(D)--(E)--(A)--(B)--(C)--(D)--(G);
\draw (B)--(E);
\draw [dashed] (D)--(B)--(F)--(D)--(A)--(C)--(E)--(G)--(C);
\node  at (90:1.25cm) {$a$};
\node  at (162:1.25cm) {$b$};
\node  at (225:1.25cm) {$c$};
\node  at (315:1.25cm) {$d$};
\node  at (378:1.25cm) {$e$};
\node  at (234:2.25cm) {$f$};
\node  at (306:2.25cm) {$g$};
\end{tikzpicture}
&&
\begin{tikzpicture}[scale=1]
\node [circle,fill=black,inner sep=1pt] (A) at (90:1cm) {};
\node [circle,fill=black,inner sep=1pt] (B) at (162:1cm) {};
\node [circle,fill=black,inner sep=1pt] (C) at (234:1cm) {};
\node [circle,fill=black,inner sep=1pt] (D) at (306:1cm) {};
\node [circle,fill=black,inner sep=1pt] (E) at (378:1cm) {};
\node [circle,fill=black,inner sep=1pt] (F) at (234:2cm) {};
\node [circle,fill=black,inner sep=1pt] (G) at (306:2cm) {};
\draw (G)--(F)--(C)--(D)--(E)--(A)--(B)--(C)--(D);
\draw (B)--(E);
\draw [dashed] (E)--(C)--(A)--(D)--(B)--(F)--(D);
\draw [dashed] (C)--(G);
\node  at (90:1.25cm) {$a$};
\node  at (162:1.25cm) {$b$};
\node  at (225:1.25cm) {$c$};
\node  at (315:1.25cm) {$d$};
\node  at (378:1.25cm) {$e$};
\node  at (234:2.25cm) {$f$};
\node  at (306:2.25cm) {$g$};
\end{tikzpicture}\\
Figure 4.4.1&&Figure 4.4.2&&Figure 4.4.3
\end{tabular}
\end{center}
\begin{lemma}\label{disjoint-triangles:C7}
The graph $\G$ does not have any heptagon.
\end{lemma}
\begin{proof}
Suppose on the contrary that $\G$ has a heptagon $S$ with vertices $a,b,c,d,e,f,g$. By Lemmas \ref{disjoint-triangles:C5} and \ref{disjoint-triangles:C6}, $S$ is an induce subgraph. Since $\G$ is not an odd cycle, there is a new vertex $h$ adjance to some vertex of $S$. A simple verification shows that $h$ is adjacent to two consecutive vertices of $S$ in $\G$ or $\G_2$. Hence, we may assume that $h$ is adjacent to vertices $d$ and $e$ of $S$ in $\G$. By Lemmas \ref{disjoint-triangles:C6} and \ref{disjoint-triangles:C5}, one gets $N_S(h)=\{d,e\}$. By the same reasons, one can easily show that the subgraph of $\G_2$ induced by the vertices $a,b,c,d,e,f,g,h$ is as drawn in Figure 4.5.1 with dashed lines. But then $(\G_2)_2$ has two triangles $\{a,e,h\}$ and $\{a,d,h\}$ with a common edge, which is a contradiction.
\end{proof}
\begin{center}
\begin{tikzpicture}[scale=1]
\node [circle,fill=black,inner sep=1pt] (A) at (90+360/7*0:1cm) {};
\node [circle,fill=black,inner sep=1pt] (B) at (90+360/7*1:1cm) {};
\node [circle,fill=black,inner sep=1pt] (C) at (90+360/7*2:1cm) {};
\node [circle,fill=black,inner sep=1pt] (D) at (90+360/7*3:1cm) {};
\node [circle,fill=black,inner sep=1pt] (E) at (90+360/7*4:1cm) {};
\node [circle,fill=black,inner sep=1pt] (F) at (90+360/7*5:1cm) {};
\node [circle,fill=black,inner sep=1pt] (G) at (90+360/7*6:1cm) {};
\node [circle,fill=black,inner sep=1pt] (H) at (270:1.8019cm) {};
\draw (A)--(B)--(C)--(D)--(E)--(F)--(G)--(A);
\draw [dashed] (A)--(C)--(E)--(G)--(B)--(D)--(F)--(A);
\draw (D)--(H)--(E);
\draw [dashed] (C) to [out=270,in=180] (H);
\draw [dashed] (F) to [out=270,in=0] (H);
\node  at (90+360/7*0:1.25cm) {$a$};
\node  at (90+360/7*1:1.25cm) {$b$};
\node  at (90+360/7*2:1.25cm) {$c$};
\node  at (90+360/7*3:1.25cm) {$d$};
\node  at (90+360/7*4:1.25cm) {$e$};
\node  at (90+360/7*5:1.25cm) {$f$};
\node  at (90+360/7*6:1.25cm) {$g$};
\node  at (270:2.0cm) {$h$};
\end{tikzpicture}\\
Figure 4.5.1
\end{center}
\begin{lemma}\label{disjoint-triangles:C8}
The graph $\G$ does not have any octagon.
\end{lemma}
\begin{proof}
Suppose on the contrary that $\G$ has an octagon $S$ with $a,b,c,d,e,f,g,h$ as its vertices. By Lemmas \ref{disjoint-triangles:C5}, \ref{disjoint-triangles:C6} and \ref{disjoint-triangles:C7}, $S$ is an induced subgraph of $\G$. Since $\G$ is not an even cycle, there is a new vertex $i$ adjacent to some vertex of $S$. Clearly, $i$ is adjacent to two consecutive vertices of $S$ for otherwise we have a pentagon in $\G_2$ contradicting Lemma \ref{disjoint-triangles:C5}. Hence, we may assume that $i$ is adjacent to vertices $d$ and $e$ of $S$. Now, by Lemma \ref{disjoint-triangles:C5}, one can easily show that the subgraph of $\G_2$ induced by the vertices $a,b,c,d,e,f,g,h,i$ is as drawn in Figure 4.6.1. But then $i$ is adjacent to vertices $a,d,e,h$ in $(\G_2)_2$ contradicting Lemma \ref{disjoint-triangles:Delta}. The proof is complete.
\end{proof}
\begin{center}
\begin{tikzpicture}[scale=1]
\node [circle,fill=black,inner sep=1pt] (A) at (112.5+45*0:1cm) {};
\node [circle,fill=black,inner sep=1pt] (B) at (112.5+45*1:1cm) {};
\node [circle,fill=black,inner sep=1pt] (C) at (112.5+45*2:1cm) {};
\node [circle,fill=black,inner sep=1pt] (D) at (112.5+45*3:1cm) {};
\node [circle,fill=black,inner sep=1pt] (E) at (112.5+45*4:1cm) {};
\node [circle,fill=black,inner sep=1pt] (F) at (112.5+45*5:1cm) {};
\node [circle,fill=black,inner sep=1pt] (G) at (112.5+45*6:1cm) {};
\node [circle,fill=black,inner sep=1pt] (H) at (112.5+45*7:1cm) {};
\node [circle,fill=black,inner sep=1pt] (I) at (270:1.8477cm) {};
\draw (A)--(B)--(C)--(D)--(E)--(F)--(G)--(H)--(A);
\draw [dashed] (A)--(C)--(E)--(G)--(A);
\draw [dashed] (B)--(D)--(F)--(H)--(B);
\draw (D)--(I)--(E);
\draw [dashed] (C) to [out=270,in=180] (I);
\draw [dashed] (F) to [out=270,in=0] (I);
\node  at (112.5+45*0:1.25cm) {$a$};
\node  at (112.5+45*1:1.25cm) {$b$};
\node  at (112.5+45*2:1.25cm) {$c$};
\node  at (112.5+45*3:1.25cm) {$d$};
\node  at (112.5+45*4:1.25cm) {$e$};
\node  at (112.5+45*5:1.25cm) {$f$};
\node  at (112.5+45*6:1.25cm) {$g$};
\node  at (112.5+45*7:1.25cm) {$h$};
\node  at (270:2.1cm) {$i$};
\end{tikzpicture}\\
Figure 4.6.1
\end{center}
\begin{theorem}\label{disjoint-triangles}
Let $\G$ be a self $2$-distance graph with disjoint triangles. Then either $\G$ is an odd cycle or it is the edged product $C_5|C_3$.
\end{theorem}
\begin{proof}
A simple verification shows that squares in $\G_2$ arises from hexagons or octagons. Hence, by Lemmas \ref{disjoint-triangles:C6} and \ref{disjoint-triangles:C8}, $\G$ has no squares and the result follows by Theorem \ref{C4free}.
\end{proof}
\begin{corollary}\label{cubic}
There is no cubic self $2$-distance graph.
\end{corollary}
\begin{proof}
By Theorem \ref{disjoint-triangles}, and the fact that $\G$ is not the complete graph on four vertices, it follows that $\G$ has an induced subgraph as in Figure 4.8.1. Then $\deg_{\G_2}(u)$ is two, which is a contradiction.
\end{proof}
\begin{center}
\begin{tikzpicture}
\node [circle,fill=black,inner sep=1pt] (A) at (0,0) {};
\node [circle,fill=black,inner sep=1pt] (B) at (1,0) {};
\node [circle,fill=black,inner sep=1pt,label=above:\tiny{$u$}] (C) at ({1+cos(45)},{sin(45)}) {};
\node [circle,fill=black,inner sep=1pt] (D) at ({1+cos(45)},{-sin(45)}) {};
\node [circle,fill=black,inner sep=1pt] (E) at ({1+2*cos(45)},0) {};
\node [circle,fill=black,inner sep=1pt] (F) at ({2+2*cos(45)},0) {};

\draw (A)--(B);
\draw (E)--(F);
\draw (C)--(D)--(B)--(C)--(E)--(D);
\end{tikzpicture}\\
Figure 4.8.1
\end{center}
\section{Graphs with no diamond subgraphs}
In this section, we go further into the study of self $2$-distance graphs with a forbidden subgraph, which relies on our earlier results. Remind that a diamond is the edged product of two triangles, namely $C_3|C_3$. A diamond with vertices $a,b$ of degree three and vertices $c,d$ of degree two is denoted by $\D(a,b;c,d)$.
\begin{theorem}\label{diamond-free}
Let $\G$ be a self $2$-distance graph with no diamond as subgraph. Then either $\G$ is an odd cycle, it is the edged product $C_5|C_3$, or it is isomorphic to one the following graphs:
\begin{center}
\begin{tabular}{ccc}
\begin{tikzpicture}[scale=0.75]
\node [circle,fill=black,inner sep=1pt] (A) at (2,2) {};
\node [circle,fill=black,inner sep=1pt] (B) at (2,-2) {};
\node [circle,fill=black,inner sep=1pt] (C) at (-2,-2) {};
\node [circle,fill=black,inner sep=1pt] (D) at (-2,2) {};
\node [circle,fill=black,inner sep=1pt] (E) at (1,0) {};
\node [circle,fill=black,inner sep=1pt] (F) at (0,-1) {};
\node [circle,fill=black,inner sep=1pt] (G) at (-1,0) {};
\node [circle,fill=black,inner sep=1pt] (H) at (0,1) {};
\draw (A)--(B)--(C)--(D)--(A)--(E)--(B)--(F)--(C)--(G)--(D)--(H)--(A);
\draw (E)--(G);
\draw (F)--(H);
\end{tikzpicture}
&\ \hspace{1cm}\ &
\begin{tikzpicture}[scale=0.75]
\node [circle,fill=black,inner sep=1pt] (A) at (2,2) {};
\node [circle,fill=black,inner sep=1pt] (B) at (2,-2) {};
\node [circle,fill=black,inner sep=1pt] (C) at (-2,-2) {};
\node [circle,fill=black,inner sep=1pt] (D) at (-2,2) {};
\node [circle,fill=black,inner sep=1pt] (E) at (1,0) {};
\node [circle,fill=black,inner sep=1pt] (F) at (0,-1) {};
\node [circle,fill=black,inner sep=1pt] (G) at (-1,0) {};
\node [circle,fill=black,inner sep=1pt] (H) at (0,1) {};
\node [circle,fill=black,inner sep=1pt] (O) at (0,0) {};
\draw (A)--(B)--(C)--(D)--(A)--(E)--(B)--(F)--(C)--(G)--(D)--(H)--(A);
\draw (E)--(G);
\draw (F)--(H);
\draw (E) to [out=135,in=45] (G);
\draw (F) to [out=135,in=225] (H);
\end{tikzpicture}\\
Figure 5.1.1&&Figure 5.1.2
\end{tabular}
\end{center}
\end{theorem}
\begin{proof}
First we observe that $\Delta(\G)\leq4$. Indeed, if $v\in V(\G)$ is an arbitrary vertex, then by assumption the subgraph induced by $N_\G(v)$ is a union of disjoint edges and isolated vertices. On the other hand, $N_\G(v)^c$ is a subgraph of $\G_2$, from which it follows that $|N_\G(v)|\leq4$. If $\Delta(\G)\leq3$, then all triangles in $\G$ are disjoint and the result follows by Theorem \ref{disjoint-triangles}. Hence, in what follows, we assume that $\Delta(\G)=4$ and that $v\in V(\G)$ is a vertex of degree four. Clearly, $N_\G(v)$ is a union of two disjoint edges, say $\{a,b\}$ and $\{c,d\}$, and that $N_\G(a)\cap N_\G(b)=N_\G(c)\cap N_\G(d)=\{v\}$. Let $X=\{a,b\}$, $Y=\{c,d\}$ and $M_\G(v)$ be the set of all vertices of $\G\setminus\{v\}$ adjacent to an element of $X$ and an element of $Y$. Suppose further that $|M_\G(v)|$ is maximum among all vertices of degree four. We proceed in some steps:

Case 1. If $e,f\in M_\G(v)$, then $N_{N_\G(v)}(e)\neq N_{N_\G(v)}(f)$. If not $(N_\G(v)\cup\{e,f\})\setminus N_{N_\G(v)}(e)$ has a diamond subgraph in $\G_2$, which is a contradiction.

Case 2. If $e,f\in M_\G(v)$, then $N_\G(e),N_\G(f)\subseteq N_\G(v)$. If not we may assume that $e$ is adjacent to a new vertex $g$. First assume that $N_{N_\G(v)}(e)\cap N_{N_\G(v)}(f)=\emptyset$ and we may assume that $a,c\in N_\G(e)$ and $b,d\in N_\G(f)$. Then we have a diamond $\D(e,g;a,c)$ in $\G$ or a diamond $\D(a,c;f,g)$ in $\G_2$ according to $g$ is adjacent simultaneously to $a$ and $c$ or not, which is a contradiction. Hence $g$ is adjacent to exactly one of $a$ or $c$, say $a$. Hence $g$ is not adjacent to $d$ by assumption, from which it follows that $g$ and $f$ are not adjacent for otherwise we get a diamond $\D(b,d;e,g)$ in $\G_2$. Now, by replacing $\G,v,a,b,c,d,e,f,g$ by $\G_2,b,c,g,d,e,a,v,f$, we observe that $g$ and $f$ are adjacent, which is impossible. Thus $N_{N_\G(v)}(e)\cap N_{N_\G(v)}(f)\neq\emptyset$. Assume that $a,c\in N_\G(e)$ and $a,d\in N_\G(f)$. Then $\deg_\G(a)=4$, which implies that $e$ and $f$ are adjacent. Hence $\deg_\G(e)=4$ so that $c$ and $g$ are adjacent too. Now, by replacing $\G,v,a,b,c,d,e,f$ by $\G_2,b,e,d,f,c,v,a$, respectively, we observe that $N_{N_\G(v)}(e)\cap N_{N_\G(v)}(f)=\emptyset$, which is impossible as mentioned before.

Case 3. $|M_\G(v)|=4$. We may assume that $N_\G(a)\cap N_\G(c)=\{v,e\}$, $N_\G(b)\cap N_\G(d)=\{v,f\}$, $N_\G(b)\cap N_\G(c)=\{v,g\}$ and $N_\G(a)\cap N_\G(d)=\{v,h\}$ for some distinct vertices $e,f,g,h$ different from $v,a,b,c,d$. As $\deg_\G(a)=\deg_\G(b)=\deg_\G(c)=\deg_\G(d)=4$, the subgraph induced by $e,f,g,h$ is the $4$-cycle $\{e,g,f,h,e\}$. Hence, by using case 2, the graph $\G$ is isomorphic to that drawn in Figure 5.1.2.

Case 4. $|M_\G(v)|=3$. We may assume that $N_\G(a)\cap N_\G(c)=\{v,e\}$, $N_\G(b)\cap N_\G(d)=\{v,f\}$, $N_\G(b)\cap N_\G(c)=\{v,g\}$ for some distinct vertices $e,f,g$ different from $v,a,b,c,d$. Since $\deg_\G(b)=\deg_\G(c)=4$, $g$ is adjacent to $e$ and $f$. But then $e$ and $f$ are not adjacent for otherwise we obtain a diamond $\D(e,g;c,f)$. If $\G$ has more than eight vertices, then there exists a new vertex $h$ adjacent to $a$ or $d$, say $a$. Since $\deg_\G(a)=4$, $e$ and $h$ must be adjacent. Then $h$ is adjacent to $b,c,v$ in $\G_2$ so that $|M_{\G_2}(v)|=4$. Hence, by case 3, $\G$ is isomorphic to the graph in Figure 5.1.2, which is a contradiction. Therefore, the only vertices of $\G$ are $v,a,b,c,d,e,f,g$ and $\G$ is isomorphic to the graph drawn in Figure 5.1.1.

Case 5. $|M_\G(v)|=2$. Then $M_\G(v)=\{e,f\}$ for some vertices $e$ and $f$. First assume that $N_{N_\G(v)}(e)\cap N_{N_\G(v)}(f)=\emptyset$, say $N_{N_\G(v)}(e)=\{a,c\}$ and $N_{N_\G(v)}(f)=\{b,d\}$. By case 2, there exists a new vertex $g$ adjacent to $a,b,c$ or $d$, say $a$. Then $\deg_\G(a)=4$, which implies that $g$ and $e$ are adjacent, contradicting case 2.
Thus $N_{N_\G(v)}(e)\cap N_{N_\G(v)}(f)\neq\emptyset$, say $N_{N_\G(v)}(e)=\{a,c\}$ and $N_{N_\G(v)}(f)=\{a,d\}$. Then $M_{\G_2}(b)=\{a,v\}$ and $N_{N_\G(b)}(a)\cap N_{N_\G(b)}(v)=\emptyset$, which is impossible by similar arguments as before.

Case 6. $|M_\G(v)|=1$. Suppose that $M_\G(v)=\{e\}$ and $N_{N_\G(v)}(e)=\{a,c\}$. First, we observe that neither $a$ nor $c$ is adjacent to a new vertex. If not we may assume that $a$ is adjacent to a new vertex $f$, from which it follows $e$ and $f$ must be adjacent. But then $a,v\in M_{\G_2}(b)$ contradicting the choice of $v$ as $\G_2\cong\G$. Now, if two new vertices $f$ and $g$ are adjacent to $b$ or $d$, say $b$, then $\deg_{\G_2}(a)=4$ while $b$ is an isolated vertex in $N_{\G_2}(a)=\{c,d,f,g\}$, which is a contradiction. Hence, we may assume that neither $b$ nor $d$ is adjacent to two vertices other than $v,a,c$. Next assume that $b$ and $d$ are adjacent to new vertices $f$ and $g$, respectively. If $f$ and $g$ are adjacent, then $a,v\in M_{\G_2}(b)$, contradicting the choice of $v$. Hence, assume that $f$ and $g$ are not adjacent and consequently $b$ and $d$ are not adjacent to $g$ and $f$ in $\G_2$, respectively. Also, $a$ and $g$ are not adjacent in $\G_2$ for otherwise $d$ and $f$ must be adjacent in $\G_2$, which is impossible. Now, it is easy to see that $(\G_2)_2$ has a diamond $\D(a,b;g,v)$, which is a contradiction. Hence, we may assume that at most one of $b$ and $d$ are adjacent to a new vertex. Suppose $b$ is such an element adjacent to a new vertex $f$. Then either we have a diamond $\D(c,d;f,v)$ in $(\G_2)_2$ when $c$ and $f$ are not adjacent in $\G_2$, or $e,f\in M_{(\G_2)_2}(v)$ when $c$ and $f$ are adjacent in $\G_2$, which is a contradiction. Therefore,  neither $b$ nor $d$ is adjacent to a new vertex other than $v,a,c$. Then, the second neighborhood of $v$ is consists of $e$ only. Since the subgraph induced by $a,b,c,d,e$ is not self $2$-distance graph, the vertex $e$ must be adjacent to some vertices other than $v,a,b,c,d$. If $e$ is adjacent to two new vertices $f,g$, then $\G_2$ has a diamond $\D(a,c;f,g)$, which is a contradiction. Hence, $N_\G(e)=\{a,c,f\}$ for some vertex $f$. As $b,d,v\in N_{\G_2}(e)$ and $\deg_{\G_2}(e)\leq4$, there must exists another vertex $g$ such that $N_\G(f)=\{e,g\}$. Then $N_{\G_2}(e)=\{b,d,v,g\}$ so that $v$ and $g$ must be adjacent in $\G_2$, which is impossible as $d_\G(v,g)=4$. 

Case 7. $M_\G(v)=\emptyset$. First suppose that three vertices among $a,b,c,d$ are adjacent to new vertices, say $a,b,c$ are adjacent to distinct vertices $e,f,g$, respectively. If $g$ is adjacent to $e$ or $f$, say $e$, then $N_{\G_2}(a)=\{c,d,f,g\}$ and hence $c$ and $f$ must be adjacent in $\G_2$, that is, $c$ and $f$ are connected in $\G$ via a path of length $2$. Clearly, $f$ and $g$ are not adjacent  for otherwise we have a diamond $\D(d,g;a,b)$ in $\G_2$, a contradiction. Hence, there exists a new vertex $h$ adjacent to both $c$ and $f$. Then $N_\G(c)=\{v,d,g,h\}$ so that $g$ and $h$ must be adjacent. But then $f$ and $g$ are adjacent in $\G_2$, which results in a diamond $\D(a,f;c,g)$ in $\G_2$, a contradiction. Thus, we deduce that there is no edges from $N_\G(a)\cup N_\G(b)$ to $N_\G(c)\cup N_\G(d)$, from which we obtain a diamond $\D(a,b;v,g)$ in $(\G_2)_2$, a contradiction. Next assume that exactly two vertices among $a,b,c,d$ are adjacent to vertices other than $v,a,b,c,d$. We have two cases up to symmetry:

(i) $a$ and $b$ are adjacent to two distinct new vertices $e$ and $f$, respectively. If $e$ or $f$, say $e$, is adjacent to another vertex $g$ in the third neighborhood of $v$, then $N_{\G_2}(a)=\{c,d,f,g\}$ where $\{c,d,f\}$ induces an independent set in $\G_2$, a contradiction. On the other hand, if $a$ or $b$, say $a$, is adjacent to another vertex $g$, then $N_{\G_2}(b)=\{c,d,e,g\}$ with $\{c,d,e\}$ an independent set in $\G_2$, which is again a contradiction.

(ii) $a$ and $c$ are adjacent to two distinct new vertices $e$ and $f$, respectively. If $e$ and $f$ are adjacent, then $e,f\in M_{(\G_2)_2}(v)$, which contradicts the choice of $v$ as $(\G_2)_2\cong\G$. Hence, we may assume that there is no edges from $N_\G(a)\setminus\{v,b\}$ to $N_\G(c)\setminus\{v,d\}$. If $e$ is adjacent to a new vertex $g$, then $N_{(\G_2)_2}(d)=\{c,e,g,v\}$ with $\{c,g,v\}$ an independent set in $(\G_2)_2$, which is impossible. Hence, $N_\G(e)=\{a\}$ and similarly $N_\G(f)=\{c\}$. Since, the subgraph induced by $v,a,b,c,d,e,f$ is not a self $2$-distance graph, we may assume that $a$ is adjacent to another vertex $g$. But then $e$ and $g$ are adjacent and hence $N_{\G_2}(b)=\{c,d,e,g\}$ with $\{d,e,g\}$ an independent subset in $\G_2$, which is a contradiction.

Finally, assume that only one of the vertices $a,b,c,d$ is adjacent to a vertex other than $v,a,b,c,d$, say $a$ is adjacent to a new vertex $e$. If $a$ is adjacent to another vertex $f$, then as before $N_{\G_2}(b)=\{c,d,e,f\}$ with $\{d,e,f\}$ an independent subset in $\G_2$, which is a contradiction. Hence $N_\G(a)=\{v,b,e\}$ so that $e$ is adjacent to a vertex $f$, from which we obtain a diamond $\D(c,d;e,f)$ in $(\G_2)_2$, which is a contradiction. The proof is complete.
\end{proof}
\section{Open problems}
We devote the last section of this paper to some open problems arising in our study of self $2$-distance graphs. The following conjecture, if it is true, can be applied to shorten our proofs, and also will be useful while studying self $2$-distance graphs with other forbidden subgraphs.
\begin{conjecture}
Every self $2$-distance graph is $2$-connected.
\end{conjecture}

A graph $\G$ with $v$ vertices is \textit{strongly regular} of degree $k$ if there are integers $\lambda$ and $\mu$ such that every two adjacent vertices have $\lambda$ common neighbors and every two non-adjacent vertices have $\mu$ common neighbors. The numbers $v,k,\lambda,\mu$ are the parameters of the corresponding graph.
\begin{theorem}
Every strongly regular self $2$-distance graphs is a self-complimentary graph and has parameters $(4t+1,2t,t-1,t)$ where the number of vertices is a sum of two squares.
\end{theorem}
\begin{proof}
The result follows from \cite{jjs} and the fact that every strongly regular graph has diameter at most two.
\end{proof}

We have shown, in Corollary \ref{cubic}, that there is no self $2$-distance cubic graph. Indeed, we believe that a more general case also holds for regular graphs with odd degrees while the same result cannot hold for regular graphs of even degrees by the above theorem.
\begin{conjecture}
There are no regular self $2$-distance graphs of odd degree.
\end{conjecture}

We note that if the above conjecture is true then for any finite group $G$ and any inversed closed subset $S$ of $G\setminus\{1\}$ of odd size, the sets $S^2\setminus S$ and $S$ belong to different orbits of the poset of subsets of $G$ under the action of automorphism group of $G$.

\end{document}